\documentclass[final]{siamltex}
\usepackage{amssymb, amsmath}
\usepackage{float,epsfig}
\usepackage{algpseudocode}
\usepackage{algorithm}
\usepackage{enumitem}

\usepackage{tikz}
\newcommand{\Arrow}[1]{%
	\parbox{#1}{\tikz{\draw[->](0,0)--(#1,0);}}
}

\newcommand{\rsss}{\rotatebox[]{90}{$\boxminus$}\kern-0.7em{\mathrel{\raisebox{.2ex}{\Arrow{.35cm}}}}\!\!}

\newcommand{\csss}{\text{$\boxminus\kern-0.655em{\mathrel{\raisebox{-.2ex}{\rotatebox[]{-90}{\Arrow{.34cm}}}}}$\ }}

\newtheorem{remark}[theorem]{ Remark}
\newtheorem{exam}[theorem]{\bf Example}

\newcommand{\ba}{\begin{array}}
\newcommand{\ea}{\end{array}}

\newcommand{\be}{\begin{equation}}
\newcommand{\ee}{\end{equation}}
\newcommand{\beano}{\begin{eqnarray*}}
\newcommand{\eeano}{\end{eqnarray*}}

\def\lam{\lambda}

\usepackage[us,12hr]{datetime}

\title{Linearizations of Quadratic Two Parameter
matrix Polynomial Via Newton Basis}
\author{niku.namita }
\date{July 2024}

\author{
Avisek Bist \thanks{Department of Mathematics, Sikkim University, Sikkim-737102, India, ({\tt avisek.bista@gmail.com})}  \and
Namita Behera \thanks{Corresponding author\\ Department of Mathematics, Sikkim University, Sikkim-737102, India, ({\tt nbehera@cus.ac.in}, niku.namita@gmail.com).} 
}

\begin{document}

\maketitle

\begin{abstract}
Given a quadratic two-parameter matrix polynomial in Newton basis $Q_{N}(\lam, \mu)$, we construct a vector space of linear two-parameter matrix polynomials and identify a set of linearizations which lie in the vector space. We also describe construction of each of these linearizations. 
\end{abstract}



\section{Introduction}
We consider two-parameter quadratic matrix polynomials of the form
\begin{equation}
Q(\lambda, \mu) = \lam^2 A_{20}+ \lam\mu A_{11} + \mu^2 A_{02}+\lam A_{10}+ \mu A_{01}+A_{00},
\label{eq:quad_poly}
\end{equation}
where $\lambda, \mu$ are scalars and the coefficient matrices are real or complex matrices of order $n \times n$. If $(\lambda,\mu) \in \mathbb{C} \times \mathbb{C}$ and nonzero $x \in \mathbb{C}^n$ satisfy $Q(\lambda,\mu)x = 0$, then $x$ is said to be an eigenvector of $Q(\lambda, \mu)$ corresponding to the eigenvalue $(\lambda, \mu)$. 

The classical approach to solving spectral problems for matrix polynomials is to first perform a \textit{linearization}, that is, to transform the given polynomial into a linear matrix polynomial, and then work with this linear polynomial (see \cite{VA72, VBKha97, TKosir94, VNKub98, AMBPle10} and the references therein). Therefore, given a quadratic two-parameter matrix polynomial $Q(\lambda, \mu)$, we seek linear two-parameter matrix polynomials 
\[
L(\lambda, \mu) = \lambda L_{1} + \mu L_{2} + L_{0},
\]
called \textit{linearizations}, which have the same spectral properties as $Q(\lambda, \mu)$.

In \cite{BA12} the pencil $C(\lam, \mu)$ is given by 
\begin{equation}\label{cfq}
\left(
\lambda 
\begin{pmatrix}
A_{20} & A_{11} & 0 \\
0 & 0 & 0 \\
0 & 0 & I
\end{pmatrix}
+ \mu 
\begin{pmatrix}
0 & A_{02} & 0 \\
0 & 0 & I \\
0 & 0 & 0
\end{pmatrix}
+ 
\begin{pmatrix}
A_{10} & A_{01} & A_{00} \\
0 & -I & 0 \\
-I & 0 & 0
\end{pmatrix}
\right)
\begin{pmatrix}
x_{10} \\
x_{01} \\
x_{00}
\end{pmatrix}
= 0
\end{equation}

\noindent
\hfill$\underbrace{\hspace{3.8cm}}_{\displaystyle L_{1}}$\hfill
$\underbrace{\hspace{3.2cm}}_{\displaystyle L_{2}}$\hfill
$\underbrace{\hspace{3.2cm}}_{\displaystyle L_{0}}$

\vspace{.5cm}

\noindent
We refer $C(\lam, \mu)$ as companion pencil/standard pencil of $Q(\lam, \mu). $
Observe that
\[
\begin{pmatrix}
x_{10} \\
x_{01} \\
x_{00}
\end{pmatrix}
=
\begin{pmatrix}
\lambda x \\
\mu x \\
x
\end{pmatrix}
=
\begin{pmatrix}
\lambda \\
\mu \\
1
\end{pmatrix}
\otimes x.
\]
We denote $\Lambda := \begin{pmatrix} \lambda \\ \mu \\ 1 \end{pmatrix}$. Thus, $x$ is the eigenvector corresponding to an eigenvalue $(\lambda, \mu)$ of $Q(\lambda, \mu)$ if and only if 
\[
C(\lambda, \mu) w = 0,
\quad \text{where} \quad
w = \Lambda \otimes x,
\]
and $C(\lambda, \mu) = \lambda L_{1} + \mu L_{2} + L_{0}$. That is, $w$ is the eigenvector corresponding to an eigenvalue $(\lambda, \mu)$ of $C(\lambda, \mu)$.

One-parameter matrix polynomials have been a well-studied topic in numerical linear algebra \cite{TDM, effenberger2012chebyshev, HFPSorthop,  gohberg82, FGanttheorymatr98, mmmm06}. In particular, Mackey et al. \cite{mmmm06} extensively examined the one-parameter polynomial eigenvalue problem and developed vector spaces of linearizations by generalizing companion forms associated with one-parameter matrix polynomials. Following a similar line of reasoning, in \cite{BA12} they constructed a vector space of linear two-parameter matrix polynomials corresponding to a given quadratic two-parameter matrix polynomial. They provided a detailed characterization of each of these linear polynomials. 

Recently, in \cite{VPDSM18}, for a nonlinear eigenvalue problem of the form \( P(\lambda)x = 0 \), where \( x \neq 0 \) and \( P(\lambda) \) is a matrix polynomial expressed as
\[
P(\lambda) = \sum_{i=0}^k A_i \, \phi_i(\lambda),
\]
with \( B = \{\phi_i(\lambda)\}_{i=0}^k \) denoting a polynomial basis for the space of univariate scalar polynomials of degree at most \( k \). Classical examples of such bases include Chebyshev, Newton, Hermite, Lagrange, and Bernstein. Matrix polynomials expressed in these bases commonly arise either directly from applications or as approximations to more general nonlinear eigenvalue problems---see, for instance, \cite{defez2002matrix, effenberger2012chebyshev, guttel2014nleigs, vanbeeumen2013rational, vanbeeumen2015linearization} and references therein.

\begin{definition}\cite{BA12}
A $ln \times ln$ linear matrix polynomial 
\[
L(\lambda, \mu) = \lambda L_{1} + \mu L_{2} + L_{0}
\]
is a \textit{linearization} of an $n \times n$ matrix polynomial $Q(\lambda, \mu)$ if there exist matrix polynomials $P(\lambda, \mu)$ and $R(\lambda, \mu)$, whose determinants are non-zero constants independent of $\lambda$ and $\mu$, such that
\[
\begin{pmatrix}
Q(\lambda, \mu) & 0 \\
0 & I_{(l-1)n}
\end{pmatrix}
= P(\lambda, \mu) L(\lambda, \mu) R(\lambda, \mu).
\]
\end{definition}

It is shown in that $C(\lam, \mu)$ is a linearization of $Q(\lam, \mu). $ Other than this they have constructed two vector spaces linearizations for $(\lam, \mu)$ and its charactrizations.



The main contributions of this paper are as follows. 
We define
quadratic two-parameter matrix polynomials in the Newton basis and examine their linearizations.  Further, we study vector spaces linearizations of a quadratic two parameter matrix polynomial in Newton basis. Furthermore, we discuss structure preserving quadratic two parameter matrix polynomials.

In \cite{VPDSM18}, matrix polynomials expressed in the Newton basis, along with the associated polynomial eigenvalue problems, are discussed. In the next section, we define quadratic two-parameter matrix polynomials in the Newton basis and examine their linearizations.

\section{QTEP expressed in Newton Basis}
Consider the two-parameter quadratic eigenvalue problem (QTEP):
$Q(\lambda, \mu)x = 0,$
where
$$
Q(\lambda, \mu) = \sum_{i+j \leq 2} A_{ij} \lambda^i \mu^j,$$
with matrix coefficients $ A_{ij} \in \mathbb{C}^{n \times n} )$.

\subsection*{Newton Basis}
Let $\mathcal{A} = (\alpha_1, \alpha_2)$ and $\mathcal{B}=(\beta_1, \beta_2)$ be an ordered list of elements from $\mathbb{C}$, where the $\alpha$'s and $\beta$'s  need not be distinct, or numerically ordered in any way. Associated with such lists $\mathcal{A}$ and $\mathcal{B}$, we define the scalar polynomias as follows:   
Let \( \{n_0(\lambda), n_1(\lambda), n_2(\lambda)\} \) be the Newton basis in \(\lambda\), defined as:
\begin{align*}
n_0(\lambda) &= 1, \\
n_1(\lambda) &= \lambda - \alpha_1, \\
n_2(\lambda) &= (\lambda - \alpha_1)(\lambda - \alpha_2),
\end{align*}
and similarly, define \( \{m_0(\mu), m_1(\mu), m_2(\mu)\} \) in \(\mu\) as
\begin{align*}
m_0(\mu) &= 1, \\
m_1(\mu) &= \mu - \beta_1, \\
m_2(\mu) &= (\mu - \beta_1)(\mu - \beta_2).
\end{align*}

Then, the Newton basis for bivariate polynomials up to total degree $2$ consists of the six functions:
\[
\mathcal{N} =
\begin{pmatrix}
1 \\
n_1(\lambda) \\
m_1(\mu) \\
n_2(\lambda) \\
n_1(\lambda) m_1(\mu) \\
m_2(\mu)
\end{pmatrix} = \begin{pmatrix}
1 \\
(\lambda- \alpha_1) \\
(\mu-\beta_1) \\
(\lambda-\alpha_1)(\lambda-\alpha_2) \\
(\lambda - \alpha_1) (\mu -\beta_1) \\
(\mu- \beta_1) (\mu- \beta_2) 
\end{pmatrix}.
\]

We express the QTEP in the Newton basis as:
\begin{align}\label{qnb}
Q_N(\lambda, \mu) &= \sum_{i+j \leq 2} A_{ij} n_i(\lambda) m_j(\mu) \nonumber \\ & =   A_{20} n_{2}(\lam)+ A_{11}n_{1}(\lambda)m_{1}(\mu) + A_{02} m_{2}(\mu)+A_{10} n_1(\lambda)+ A_{01}m_1(\mu)+A_{00},
\end{align}
where \( A_{ij} \) are the coefficient matrices in the Newton basis.
Now, we aim to find a linearization of the two-parameter quadratic eigenvalue problem (QTEP):
\[
Q_N(\lambda, \mu)x = 0, 
\]
such that the eigenvalues \((\lambda, \mu)\) are preserved and the representation is structured in the Newton basis.

That is, we seek a linear two-parameter eigenvalue problem (LTEP) of the form: $
L_{N}(\lambda, \mu) w  = 0,$
such that:
\[
L_{N}(\lambda, \mu)w = 0 \quad \Leftrightarrow \quad Q_N(\lambda, \mu)x= 0.
\]

\subsection{Construction of pencils in Newton bases.}
Define scalar polynomials with $\mathcal{A} $ and $\mathcal{B}$ namely, $\{\gamma_1(\lambda), \gamma_2(\lambda)\}$ and $\{\widetilde{\gamma}_1(\mu), \widetilde{\gamma}_2(\mu)\}$ defined by $\gamma_1(\lambda) = \lambda-\alpha_1,  \gamma_2(\lambda) = \lambda-\alpha_2$ and $\widetilde{\gamma}_1(\mu) = \mu-\beta, \widetilde{\gamma}_2(\mu) = \mu-\beta$ respectively.

Now, the newton polynomilas can alternatively defined in the $\gamma_i$'s and $\widetilde{\gamma}_i$'s via multiplicative recurrence relation
\begin{align*}
n_0(\lambda) &= 1, \\
n_1(\lambda) &= \lambda - \alpha_1, \\
n_2(\lambda) &= (\lambda - \alpha_1)(\lambda - \alpha_2)= 
n_{1}(\lambda)\gamma_{2}(\lambda),
\end{align*}
and similarly, define \( \{m_0(\mu), m_1(\mu), m_2(\mu)\} \) in \(\mu\) as
\begin{align*}
m_0(\mu) &= 1, \\
m_1(\mu) &= \mu - \beta_1, \\
m_2(\mu) &= (\mu - \beta_1)(\mu - \beta_2) =m_{1}(\mu)\widetilde{\gamma}_{2}(\mu). 
\end{align*}

Define matrices   \begin{align} \label{gamma2}
\Gamma_2(\lambda) &= \begin{pmatrix}
   \gamma_{2}(\lambda)\otimes I_n & & \\
   & \gamma_{1}(\lambda)\otimes I_n& & \\
   & & & \gamma_{1}(\lambda) \otimes I_n\\
\end{pmatrix} \\&= \begin{pmatrix}
   (\lambda-\alpha_2)\otimes I_n & & \\
   & (\lambda-\alpha_1)\otimes I_n& & \\
   & & & (\lambda-\alpha_1)\otimes I_n\\
\end{pmatrix}    
\end{align} and 
\begin{align}\label{gamma2t}
\widetilde{\Gamma}_2(\mu) &= \begin{pmatrix}
   \widetilde{\gamma}_{1}(\mu)\otimes I_n & & \\
   & \widetilde{\gamma}_{2}(\mu)\otimes I_n& & \\
   & & & \widetilde{\gamma}_{1}(\mu)\otimes I_n\\
\end{pmatrix} \\&= \begin{pmatrix}
   (\mu-\beta_1)\otimes I_n & & \\
   & (\mu-\beta_2)\otimes I_n& & \\
   & & & (\mu-\beta_1)\otimes I_n\\
\end{pmatrix}. \end{align}
Consider QTEP in the Newton basis as:
\[
Q_N(\lambda, \mu) =  A_{20} n_{2}(\lam)+ A_{11}n_{1}(\lambda)m_{1}(\mu) + A_{02} m_{2}(\mu)+A_{10} n_1(\lambda)+ A_{01}m_1(\mu)+A_{00},
\]
where \( A_{ij} \) are the coefficient matrices in the Newton basis. 
Since $n_1(\lam) = \gamma_1{(\lambda)}m_{1}(\mu) = \widetilde{\gamma}_{1}(\mu)$,  and $n_{2}(\lam) = n_{1}(\lam) = \gamma_{2}(\lam)$, $m_{2}(\mu) = m_{1}(\mu) \widetilde{\gamma}_2(\mu)$, we have
$$Q_N(\lambda, \mu) = A_{20} n_{1}(\lam)\gamma_{2}(\lam)+ A_{11}n_{1}(\lambda)m_{1}(\mu) + A_{02} m_{1}(\mu)\widetilde{\gamma}_2(\mu)+A_{10} n_1(\lambda)+ A_{01}m_1(\mu)+A_{00}.$$

In \cite{BA12}, it has been studied the vector spaces linearizations for a given $Q(\lam, \mu)$ in monomila basis. In the next section, we study vector spaces linearizations of a quadratic two parameter matrix polynomial in Newton basis.

\section{Vector Spaces Linearizations}
Consider $Q(\lam, \mu)$ is given in (\ref{eq:quad_poly}) and the pencil $L(\lam, \mu) = \lam L_1+\mu L_2+ L_0$ of $Q(\lam, \mu)$ given in \cite{BA12}. In \cite{BA12}, they introduced the vector spaces linearizations of $Q(\lam, \mu).$ For this they introduced the notation
$$\mathcal{\gamma}_{Q}=\{v \otimes Q(\lam, \mu): v \in \mathbb{C}^3\}$$ and 
and define $$\mathbb{L}(Q(\lam, \mu)):= \{L(\lam, \mu): L(\lam, \mu)(\Lambda \otimes I_n) \in \mathcal{\gamma}_{Q}\},$$ where   $\Lambda =\begin{pmatrix}
    \lam \\
    \mu\\
    1
\end{pmatrix}$. Further, it is shown in \cite{BA12} that $\mathbb{L}(Q(\lam, \mu))$ is a vector space and almost all the pencils of $\mathbb{L}(Q(\lam, \mu))$ are linearizations. 

In this paper, we discuss vector spaces linearizations in terms of Newton basis. 

Consider Q2EP in the Newton basis as:
\[
Q_N(\lambda, \mu) =  A_{20} n_{2}(\lam)+ A_{11}n_{1}(\lambda)m_{1}(\mu) + A_{02} m_{2}(\mu)+A_{10} n_1(\lambda)+ A_{01}m_1(\mu)+A_{00},
\]
where \( A_{ij} \) are the coefficient matrices in the Newton basis.
Define $$\mathcal{\gamma}_{Q_N}=\{v \otimes Q_N(\lam, \mu): v \in \mathbb{C}^3\}$$ and the space 
$$\mathcal{N}(Q_N(\lam, \mu)):= \{L(\lam, \mu): L(\lam, \mu)(N \otimes I_n) \in \mathcal{\gamma}_{Q_N}\}, $$
where   $N(\lambda, \mu) =\begin{pmatrix}
    n_1(\lam) \\
   m_1(\mu)\\
    1
\end{pmatrix} = \begin{pmatrix}
    \lam -\alpha_1 \\
   \mu-\beta_1\\
    1
\end{pmatrix}$. 
Note that $\mathcal{N}(Q_N(\lam, \mu))$ is a vector space from the definition and the properties of Kronecker product.

\begin{proposition}
Let $Q_N(\lam, \mu)$ be an $n \times n$ Q2EVP. Then $\mathbb{L}(Q)$ and $\mathcal{N}(Q_N(\lam, \mu))$ are isomorphic as vector spaces.    
\end{proposition}

\begin{proof}
Note that $\Lambda(\lam, \mu) = \{\lam, \mu, 1\}$ and $N(\lam, \mu)= \{\lam-\alpha_1, \mu-\beta, 1\}$    and both bases for $\mathcal{P}_2$ implies that there exists a nonsingular constant change of basis $S$ st $S \Lambda(\lam, \mu) = N(\lam, \mu).$ Define a map
$$f: \mathbb{L}(Q)\longrightarrow \mathcal{N}(Q_N) \text{   by   } L(\lam, \mu) \longmapsto L(\lam, \mu)(S^{-1}\otimes I_n).$$

Suppose $L(\lam, \mu) \in \mathbb{L}(Q)$ with right ansatz vector $v$. Then
$$L(\lambda, \mu) \cdot (\Lambda \otimes I_n) = v \otimes Q(\lambda, \mu)
\;\Leftrightarrow\;
L(\lambda, \mu) \cdot (S^{-1} \otimes I_n) \cdot (S \otimes I_n) \cdot (\Lambda(\lambda, \mu) \otimes I_n) = v \otimes Q(\lambda, mu) $$$$
\;\Leftrightarrow\;
L(\lambda, \mu) \cdot (S^{-1} \otimes I_n) \cdot (N(\lambda) \otimes I_n) = v \otimes Q(\lambda, \mu).$$
Therefore, $ L(\lambda, \mu) \cdot (S^{-1} \otimes I_n)  \in N(Q_N)$ with right ansatz vector $v$. It shows that $f$ is a well-defined map from $\mathbb{L}(Q)$ to $\mathcal{N}(Q_N)$. It is easy to check that $f$ is also a linear map.
Next, by a completely analogous argument, one can show that the map
$$g : \mathcal{N}(Q_N) \longrightarrow \mathbb{L}(Q) \text{    by    }  T(\lam, \mu) \longmapsto (S \otimes I_n)$$
is also well defined and linear. Hence $g$ is the inverse mapping of $f$, showing that $f$ is a linear isomorphism between $\mathbb{L}(Q)$ and $\mathcal{N}(Q_N)$.

\end{proof}

Given that $\mathbb{L}$ and $\mathcal{N}$ are isomorphic as vector spaces, one might wonder why it is worth bothering with the $\mathcal{N}$ space at all, since so much is already known about $\mathbb{L}$ in \cite{BA12}. However, when $Q(\lambda, \mu)$ is expressed in the Newton basis as in equation (\ref{qnb}), it turns out to be more natural to look for linearizations in the spaces $\mathcal{N}(Q_N)$, rather than in either $\mathbb{L}(P)$ 

In particular, pencils in $\mathcal{N}$ is much easier to construct from the matrix coefficients of $Q_N(\lambda, \mu)$ than are the pencils in the $\mathbb{L}(\lam, \mu)$ space, especially if the pencils do not need to be block-symmetric.

Define $$\widetilde{\Lambda}(\lam, \mu) = \begin{pmatrix}
    \lambda^2 \\
    \lam \mu\\
    \mu^2\\
    \lambda\\
    \mu\\
    1  
\end{pmatrix} \text{  and   } \widetilde{N}(\lam, \mu) = \begin{pmatrix}
    n_2(\lambda)\\
    n_1(\lam)m_1 (\mu)\\
    m_2(\mu)\\
    n_1(\lambda)\\
    m_1(\mu\\
    n_{0} 
\end{pmatrix}$$

\begin{lemma}\label{rbaqaqn}
Let \[
Q_N(\lambda, \mu) =  A_{20} n_{2}(\lam)+ A_{11}n_{1}(\lambda)m_{1}(\mu) + A_{02} m_{2}(\mu)+A_{10} n_1(\lambda)+ A_{01}m_1(\mu)+A_{00},
\] be an $n \times n$ matrix polynomial of grade 2 in a Newton basis, and define the partner polynomial $Q(\lambda, \mu) = \lam^2 A_{20}+ \lam\mu A_{11} + \mu^2 A_{02}+\lam A_{10}+ \mu A_{01}+A_{00},$, using the same coefficients $A_{ij}$ as in $Q_N(\lambda, \mu)$. Then for matrices Then for matrices $A_1, A_2, A_3 \in  \mathbb{C}^{3n \times 3n},$ we have 
\begin{equation}
  L(\lam, \mu)\cdot (\Lambda \otimes I_n) = v \otimes Q(\lambda, \mu) \iff (A_1 \Gamma_2 + A_2 \widetilde{\Gamma}_2 +A_3) (N \otimes I_n) = v \otimes Q_N(\lambda, \mu)   
\end{equation}
where $\Gamma_2$ and $\widetilde{\Gamma}_2$ are as in (\ref{gamma2}) and (\ref{gamma2t}), respectively. Moreover, the pencils $L(\lam, \mu) = \lam A_1+\mu A_2+A_3 $ and $A_1 \Gamma_2 + A_2 \widetilde{\Gamma}_2 +A_3$ share the same
ansatz vector $v$.   
That is, $$\lam A_1+\mu A_2+A_3 \in L(Q) \iff A_1 \Gamma_2 + A_2 \widetilde{\Gamma}_2 +A_3 \in \mathcal{N}(Q_N). $$
\end{lemma}

\begin{proof}
Let $\lam A_1+\mu A_2+A_3 \in L(Q).$ Then there exists a right ansatz vector $v \in \mathbb{C}^3$ such that $(\lam A_1+\mu A_2+A_3) \cdot (\Lambda \otimes I_n) = v \otimes Q(\lam, \mu).$  By the properties of block matrix multiplication and Kronecker product we have the following:
$$(\lam A_1+\mu A_2+A_3) \cdot (\Lambda \otimes I_n) = v \otimes Q(\lam, \mu)$$
$$\lam A_1 \cdot (\Lambda \otimes I_n) + \mu A_2 \cdot (\Lambda \otimes I_n) + A_3 \cdot (\Lambda \otimes I_n) = (v.1)\otimes \left([A_{20} \,\, A_{11}\,\, A_{02} \,\, A_{10}\,\, A_{01}\,\, A_{00}] . (\widetilde{\Lambda}(\lam, \mu) \otimes I_n)\right)$$
$$\left[\begin{pmatrix}
 X_{11} & X_{12} & 0 & X_{13} & 0 & 0   
\end{pmatrix} + \begin{pmatrix}
 0 & Y_{11} & Y_{12} & 0 & Y_{13} & 0    
\end{pmatrix} + \begin{pmatrix}
 0 & 0 & 0 &Z_{11} & Z_{12} &  Z_{13}    
\end{pmatrix} \right](\widetilde{\Lambda}(\lam, \mu) \otimes I_n) $$ $$ =  \left( v\otimes [A_{20} \,\, A_{11}\,\, A_{02} \,\, A_{10}\,\, A_{01}\,\, A_{00}]\right) . (\widetilde{\Lambda}(\lam, \mu) \otimes I_n)$$
$$\left[\begin{pmatrix}
 X_{11} & X_{12} & 0 & X_{13} & 0 & 0   
\end{pmatrix} + \begin{pmatrix}
 0 & Y_{11} & Y_{12} & 0 & Y_{13} & 0    
\end{pmatrix} + \begin{pmatrix}
 0 & 0 & 0 &Z_{11} & Z_{12} &  Z_{13}    
\end{pmatrix} \right]$$ $$ =  \left( v\otimes [A_{20} \,\, A_{11}\,\, A_{02} \,\, A_{10}\,\, A_{01}\,\, A_{00}]\right), \text{  follows from Lemma 3.3, } \cite{VPDSM18}.$$
$$\left[\begin{pmatrix}
 X_{11} & X_{12} & 0 & X_{13} & 0 & 0   
\end{pmatrix} + \begin{pmatrix}
 0 & Y_{11} & Y_{12} & 0 & Y_{13} & 0    
\end{pmatrix} + \begin{pmatrix}
 0 & 0 & 0 &Z_{11} & Z_{12} &  Z_{13}    
\end{pmatrix} \right] \cdot (\widetilde{N}(\lam, \mu))$$ $$ =  \left( v\otimes [A_{20} \,\, A_{11}\,\, A_{02} \,\, A_{10}\,\, A_{01}\,\, A_{00}]\right) \cdot \widetilde{N}(\lam, \mu). $$
Now it implies that 
$$(A_1 \Gamma_2 + A_2 \widetilde{\Gamma}_2 + A_3)(N \otimes I_n) = v \otimes Q_N(\lambda, \mu), $$
where $$\Gamma_2(\lambda) = \begin{pmatrix}
   \gamma_{2}(\lambda)\otimes I_n & & \\
   & \gamma_{1}(\lambda)\otimes I_n& & \\
   & & & \gamma_{1}(\lambda) \otimes I_n\\
\end{pmatrix} $$ and 
$$\widetilde{\Gamma}_2(\mu) = \begin{pmatrix}
   \widetilde{\gamma}_{1}(\mu)\otimes I_n & & \\
   & \widetilde{\gamma}_{2}(\mu)\otimes I_n& & \\
   & & & \widetilde{\gamma}_{1}(\mu)\otimes I_n\\
\end{pmatrix}$$ and 

$$\Gamma_2(\lambda) . N  = \begin{pmatrix}
    n_2(\lambda)\\
    n_1(\lam)m_1 (\mu)\\
    n_1(\lambda) 
\end{pmatrix} \,\,\, \widetilde{\Gamma}_2(\lambda) . N  = \begin{pmatrix}
    n_1(\lam)m_1 (\mu)\\
     m_2(\mu)\\
    m_1(\mu) 
\end{pmatrix}. $$
Thus we have $A_1 \Gamma_2 + A_2 \widetilde{\Gamma}_2 +A_3 \in \mathcal{N}(Q_N).$ To see the proof of the converse part, argument is just reversible. Hence proved.
\end{proof}

Now we show that given a $Q_N$ in Newton basis as in (\ref{qnb}), when pencils from $\mathcal{N}(Q_N)$ are linearizations.






Note that not all linear two-parameter matrix polynomials in the space $\mathbb{L}(Q(\lambda, \mu))$ are linearizations of $Q(\lambda, \mu)$. For example, any $L(\lambda, \mu) \in \mathbb{L}(Q(\lambda, \mu))$ corresponding to the ansatz vector $v = 0$ is not a linearization. Thus, given a quadratic two-parameter matrix polynomial $Q(\lambda, \mu)$, in \cite{BA12} it is  identified which $L(\lambda, \mu) \in \mathbb{L}(Q(\lambda, \mu))$ are linearizations [\cite{BA12}, Theorem 2.5].

For instance for the ansatz vector $v = \alpha e_1$, where $e_1 = [1 \; 0 \; 0]^T$ and $0 \ne \alpha \in \mathbb{C}$ it is shown in \cite{BA12} that pencils in $\mathbb{L}(Q(\lambda, \mu))$ are lineraizations.

\begin{theorem} \label{QAe1} \cite{BA12}
Let $
Q(\lambda, \mu) = \lambda^2 A_{20} + \lambda\mu A_{11} + \mu^2 A_{02} + \lambda A_{10} + \mu A_{01} + A_{00}$
be a quadratic two-parameter matrix polynomial with real or complex coefficient matrices of size $n \times n$. Suppose
$
L(\lambda, \mu) = \lambda \widetilde{A}_{1} + \mu \widetilde{A}_2 + \widetilde{A}_3 \in \mathbb{L}(Q(\lambda, \mu))$
is constructed with respect to the ansatz vector $v =  e_1$, where
\[
\widetilde{A}_1 = 
\begin{pmatrix}
 e_1 \otimes A_{20} & -Y_1 +  e_1 \otimes A_{11} & -Z_1 + e_1 \otimes A_{10}
\end{pmatrix},
\]
\[
\widetilde{A}_2 = 
\begin{pmatrix}
Y_1 &  e_1 \otimes A_{02} & -Z_2 +  e_1 \otimes A_{01}
\end{pmatrix},
\]
\[
\widetilde{A}_3 = 
\begin{pmatrix}
Z_1 & Z_2 & e_1 \otimes A_{00}
\end{pmatrix},
\]
and
\[
Y_1 = 
\begin{pmatrix}
Y_{11} \\
0 \\
0
\end{pmatrix}, \quad
Z_1 = 
\begin{pmatrix}
Z_{11} \\
Z_{21} \\
Z_{31}
\end{pmatrix}, \quad
Z_2 = 
\begin{pmatrix}
Z_{12} \\
Z_{22} \\
Z_{32}
\end{pmatrix} \in \mathbb{C}^{3n \times n},
\]
with
\[
\det \begin{pmatrix}
Z_{21} & Z_{22} \\
Z_{31} & Z_{32}
\end{pmatrix} \ne 0.
\]
Then $L(\lambda, \mu)$ is a linearization of $Q(\lambda, \mu)$.
\end{theorem}

Consider $Q_{N}(\lam, \mu)$ and the space $\mathcal{N}(Q_N).$ 
We show that pencils $\mathcal{N}(Q_N)$ are lineraizations for the caseof the ansatz vector $v = e_1$, where $e_1 = [1 \; 0 \; 0]^T$.

\begin{theorem}\label{lqn}
Let \[
Q_N(\lambda, \mu) =  A_{20} n_{2}(\lam)+ A_{11}n_{1}(\lambda)m_{1}(\mu) + A_{02} m_{2}(\mu)+A_{10} n_1(\lambda)+ A_{01}m_1(\mu)+A_{00},
\] be an $n \times n$ quadratic two-parameter matrix polynomial in Newton basis, and define the partner polynomial $Q(\lambda, \mu) = \lam^2 A_{20}+ \lam\mu A_{11} + \mu^2 A_{02}+\lam A_{10}+ \mu A_{01}+A_{00},$, using the same coefficients $A_{ij}$ as in $Q_N(\lambda, \mu)$.
If $L(\lambda) = \lambda A_1 +\mu A_2 + A_3 \in \mathbb{L}(Q)$ with ansatz vector $e_1$, then:
\begin{itemize}
    \item[(a)] The matrices $A_1, A_2, A_3 \in \mathbb{C}^{3n \times 3n}$ are of the form given in Theorem~\ref{QAe1}.
    
    \item[(b)] The pencil $L_{N}(\lambda, \mu)= A_1 \Gamma_2 + A_2 \widetilde{\Gamma}_2 +A_3$
    is in $\mathbb{N}(Q_N)$ with ansatz vector $e_1$.

    \item[(c)] $L_{N}(\lambda, \mu)$ is a linearization of $Q_{N}(\lambda, \mu)$.


\end{itemize} 
\end{theorem}

\begin{proof} Part (a) directly follows from Theorem~\ref{QAe1} with $\alpha =1$ and part (b) follows from Lemma~\ref{rbaqaqn} with $v =e_1.$
Now,any linear two-parameter matrix polynomial 
$L_{N}(\lambda, \mu)= A_1 \Gamma_2 + A_2 \widetilde{\Gamma}_2 +A_3$
corresponding to the ansatz vector $v = e_1$ is of the form
\begin{align*}
L_{N}(\lambda, \mu) & =  
\begin{pmatrix}
A_{20} & -Y_{11}+A_{11} & -Z_{11} +A_{01} \\
0 & -Y_{21} & -Z_{21} \\
0 & -Y_{31} & -Z_{31}
\end{pmatrix} \begin{pmatrix}
   \gamma_{2}(\lambda)\otimes I_n & & \\
   & \gamma_{1}(\lambda)\otimes I_n& & \\
   & & & \gamma_{1}(\lambda) \otimes I_n\\
\end{pmatrix} \\
 & + 
\begin{pmatrix}
Y_{11} & A_{02} & -Z_{12}+A_{01} \\
Y_{21} & 0 & -Z_{22} \\
Y_{31} & 0 & -Z_{32}
\end{pmatrix}\begin{pmatrix}
   \widetilde{\gamma}_{1}(\mu)\otimes I_n & & \\
   & \widetilde{\gamma}_{2}(\mu)\otimes I_n& & \\
   & & & \widetilde{\gamma}_{1}(\mu)\otimes I_n\\
\end{pmatrix} \\
& +
\begin{pmatrix}
Z_{11} & Z_{12} & A_{00} \\
Z_{21} & Z_{22} & 0 \\
Z_{31} & Z_{32} & 0
\end{pmatrix} \\
& =  
\begin{pmatrix}
A_{20}\gamma_2 & -Y_{11}\gamma_1+A_{11}\gamma_1 & -Z_{11}\gamma_1 +A_{01}\gamma_1 \\
0 & -Y_{21}\gamma_1 & -Z_{21} \gamma_1\\
0 & -Y_{31}\gamma_1 & -Z_{31} \gamma_1
\end{pmatrix} + 
\begin{pmatrix}
Y_{11}\widetilde{\gamma}_{1} & A_{02}\widetilde{\gamma}_{2} & -Z_{12}\widetilde{\gamma}_{1}+A_{01} \widetilde{\gamma}_{1}\\
Y_{21}\widetilde{\gamma}_{1} & 0 & -Z_{22} \widetilde{\gamma}_{1}\\
Y_{31}\widetilde{\gamma}_{1} & 0 & -Z_{32}\widetilde{\gamma}_{1}
\end{pmatrix}\\
&+
\begin{pmatrix}
Z_{11} & Z_{12} & A_{00} \\
Z_{21} & Z_{22} & 0 \\
Z_{31} & Z_{32} & 0
\end{pmatrix}.
\end{align*}

Thus, we can rewrite $L_{N}(\lambda, \mu)$ as
\[
L_{N}(\lambda, \mu) =
\begin{pmatrix}
\widetilde{W}_1(\lambda, \mu) & \widetilde{W}_2(\lambda, \mu)& \widetilde{W}_3(\lambda, \mu) \\
 Y_{21} \widetilde{\gamma}_{1} + Z_{21} & - Y_{21}\gamma_1 + Z_{22} & - Z_{21}\gamma_1 - Z_{22}\widetilde{\gamma}_{1} \\
Y_{31}\widetilde{\gamma}_{1} + Z_{31} & - Y_{31}\gamma_1 + Z_{32} & -\lambda Z_{31} \gamma_1- \mu Z_{32}\widetilde{\gamma}_{1}
\end{pmatrix},
\]
where
\[
\begin{aligned}
\widetilde{W}_1(\lambda, \mu) &=  A_{20} \gamma_2+  Y_{11} \widetilde{\gamma}_{1}+ Z_{11}, \\
\widetilde{W}_2(\lambda, \mu) &=  -Y_{11}\gamma_1+ A_{11}\gamma_1 + A_{02}\widetilde{\gamma}_{2}+ Z_{12}\\
\widetilde{W}_3(\lambda, \mu) &=  - Z_{11}\gamma_1 + A_{10} \gamma_1- Z_{12} \widetilde{\gamma}_{1}+ A_{01} \widetilde{\gamma}_{1} + A_{00}.
\end{aligned}
\]

Define
$$
\widetilde{E}(\lambda, \mu) =
\begin{pmatrix}
n_{1}(\lambda) I_n & I_n & 0 \\
m_{1}(\mu) I_n & 0 & I_n \\
I_n & 0 & 0
\end{pmatrix}. $$ 
Then we have 
$$L_{N}(\lambda, \mu) \widetilde{E}(\lambda, \mu)=
\begin{pmatrix}
Q_{N}(\lambda, \mu) & \widetilde{W}_1(\lambda, \mu)& \widetilde{W}_2(\lambda, \mu) \\
 Y_{21} \widetilde{\gamma}_{1}n_1 - Y_{21} \gamma_{1}m_1 &  Y_{21}\widetilde{\gamma}_1 + Z_{21} & - Y_{21}\gamma_1 + Z_{22} \\
Y_{31}\widetilde{\gamma}_{1}n_1 -Y_{31} \gamma_{1}m_1 & Y_{31}\widetilde{\gamma}_{1} +Z_{31}  & - Y_{31}\gamma_1 + Z_{32} 
\end{pmatrix}.$$
Setting $Y_{21} = 0=Y_{31}$ we have $$L_{N}(\lambda, \mu) \widetilde{E}(\lambda, \mu)= \begin{pmatrix}
    Q_{N}(\lam, \mu)& \widetilde{W}(\lam, \mu)\\
    0 & Z
\end{pmatrix}, $$
where $$
Z = 
\begin{pmatrix}
Z_{21} & Z_{22} \\
Z_{31} & Z_{32}
\end{pmatrix} \in \mathbb{C}^{2n \times 2n}, \quad \widetilde{W}(\lam, \mu) = \begin{pmatrix}
    \widetilde{W}_1(\lambda, \mu)& \widetilde{W}_2(\lambda, \mu)
\end{pmatrix} \in \mathbb{C}^{n \times 2n}.
$$
Since $Z$ is nonsingular, define
\[
F(\lambda, \mu) = 
\begin{pmatrix}
I & -\widetilde{W}(\lambda, \mu) Z^{-1} \\
0 & Z^{-1}
\end{pmatrix}.
\]
Then we have 
\[
\widetilde{F}(\lambda, \mu) L_{N}(\lambda, \mu) \widetilde{E}(\lambda, \mu) =
\begin{pmatrix}
Q_{N}(\lambda, \mu) & 0 \\
0 & I_{2n}
\end{pmatrix}.
\]
Note that both $\widetilde{E}(\lambda, \mu)$ and $\widetilde{F}(\lambda, \mu)$ are unimodular matrix polynomials. Hence,
\[
\det L_{N}(\lambda, \mu) = \gamma \det Q_{N}(\lambda, \mu)
\]
for some nonzero scalar $\gamma \in \mathbb{C}$. Thus, $L_{N}(\lambda, \mu)$ is a linearization of $Q_{N}(\lambda, \mu)$.
\end{proof}

\begin{exam}
Let \[
Q_N(\lambda, \mu) =  A_{20} n_{2}(\lam)+ A_{11}n_{1}(\lambda)m_{1}(\mu) + A_{02} m_{2}(\mu)+A_{10} n_1(\lambda)+ A_{01}m_1(\mu)+A_{00},
\] be an $n \times n$ quadratic two-parameter matrix polynomial in a Newton basis, and define the partner polynomial $Q(\lambda, \mu) = \lam^2 A_{20}+ \lam\mu A_{11} + \mu^2 A_{02}+\lam A_{10}+ \mu A_{01}+A_{00},$, using the same coefficients $A_{ij}$ as in $Q_N(\lambda, \mu)$.
From (\ref{cfq}), we know that the companion pencil/ standard pencil $L(\lambda, \mu) := \lambda A_1 +\mu A_2 + A_3 $, with $A_1, A_2, A_3 \in \mathbb{C}^{3n \times 3n}$ given in equation (\ref{cfq}), is in $L(Q)$ with ansatz vector $e_1$. 
Theorem~\ref{lqn}(b) then implies that
$L_{N}(\lambda, \mu)= A_1 \Gamma_2 + A_2 \widetilde{\Gamma}_2 +A_3$
is in $\mathcal{N}(Q_N)$ with ansatz vector $e_1$.  
\end{exam}

\subsection{Construction of linearizations}
Let $Q_{N}(\lambda, \mu)$ be a quadratic 
two-parameter matrix polynomial in Newton basis and $L_{N}(\lambda, \mu) \in \mathcal{N}(Q_{N}(\lambda, \mu))$ corresponding to an ansatz vector $0 \ne v \in \mathbb{C}^3$. Then the following is a procedure for determining a set of linearizations of $Q_{N}(\lambda, \mu)$:

\begin{enumerate}
    \item Suppose $Q_{N}(\lambda, \mu)$ is a quadratic two-parameter matrix polynomial and
    \[
    L_{N}(\lambda, \mu) = L_{N}(\lambda, \mu)= A_1 \Gamma_2 + A_2 \widetilde{\Gamma}_2 +A_3 \in \mathcal{N}(Q_{N}(\lambda, \mu))
    \] corresponding to the ansatz vector $v \in \mathbb{C}^3$, i.e.,
    \[
    L_{N}(\lambda, \mu)(N \otimes I_n) = v \otimes Q_{N}(\lambda, \mu).
    \]
    
    \item Select any nonsingular matrix 
    \[
    M = 
    \begin{pmatrix}
    m_{11} & m_{12} & m_{13} \\
    m_{21} & m_{22} & m_{23} \\
    m_{31} & m_{32} & m_{33}
    \end{pmatrix}
    \]
    such that $Mv = e_1 \in \mathbb{C}^3$. A list of such matrices $M$ depending on the entries of $v$ is given in the Appendix.
  \item Apply the corresponding block-transformation $M \otimes I_n$ to $L_{N}(\lambda, \mu)$. Then we have
\[
\widehat{L}_{N}(\lambda, \mu) = (M \otimes I_n) L_{N}(\lambda, \mu) = \widehat{A}_1 \Gamma_2 + \widehat{A}_2 \widetilde{\Gamma}_2 +\widehat{A}_3,
\]
where
\begin{align*}
\widehat{A}_1 &= \left[ e_1 \otimes A_{20} \,\,\,  - \widehat{Y}_1 +  e_1 \otimes A_{11}  \,\,\,- \widehat{Z}_1 +  e_1 \otimes A_{10} \right], \\
A_2 &= \left[ \widehat{Y}_1 \quad  e_1 \otimes A_{02} - \widehat{Z}_2 +  e_1 \otimes A_{01} \right], \\
A_3 &= \left[ \widehat{Z}_1 \quad \widehat{Z}_2 \quad e_1 \otimes A_{00} \right].
\end{align*}

\[
\widehat{Y}_1 = (M \otimes I_n) Y_1 = 
\begin{pmatrix}
m_{11} Y_{11} \\
m_{21} Y_{11} \\
m_{31} Y_{11}
\end{pmatrix}, \quad
\widehat{Z}_1 = (M \otimes I_n) 
\begin{pmatrix}
Z_{11} \\
Z_{21} \\
Z_{31}
\end{pmatrix} =
\begin{pmatrix}
m_{11} Z_{11} + m_{12} Z_{21} + m_{13} Z_{31} \\
m_{21} Z_{11} + m_{22} Z_{21} + m_{23} Z_{31} \\
m_{31} Z_{11} + m_{32} Z_{21} + m_{33} Z_{31}
\end{pmatrix},
\]

\[
\widehat{Z}_2 = (M \otimes I_n) 
\begin{pmatrix}
Z_{12} \\
Z_{22} \\
Z_{32}
\end{pmatrix} =
\begin{pmatrix}
m_{11} Z_{12} + m_{12} Z_{22} + m_{13} Z_{32} \\
m_{21} Z_{12} + m_{22} Z_{22} + m_{23} Z_{32} \\
m_{31} Z_{12} + m_{32} Z_{22} + m_{33} Z_{32}
\end{pmatrix},
\]
where $\widehat{Z}_1$ and $\widehat{Z}_2$ are arbitrary.

\item For $\widehat{L}(\lambda, \mu)$ to be a linearization, we need to choose $\widehat{Y}_1$, $\widehat{Z}_1$, and $\widehat{Z}_2$ as follows. If $m_{21} = m_{31} = 0$, then choose $Y_{11}$ arbitrary; otherwise choose $Y_{11} = 0$. Further, we need to choose
\[
\widehat{Z}_1 = 
\begin{pmatrix}
Z_{11} \\
Z_{21} \\
Z_{31}
\end{pmatrix}, \quad
\widehat{Z}_2 = 
\begin{pmatrix}
Z_{12} \\
Z_{22} \\
Z_{32}
\end{pmatrix},
\]
such that
\begin{equation}
\det \begin{pmatrix}
m_{21} Z_{11} + m_{22} Z_{21} + m_{23} Z_{31} & m_{21} Z_{12} + m_{22} Z_{22} + m_{23} Z_{32} \\
m_{31} Z_{11} + m_{32} Z_{21} + m_{33} Z_{31} & m_{31} Z_{12} + m_{32} Z_{22} + m_{33} Z_{32}
\end{pmatrix} \neq 0. \label{detcond}    
\end{equation}

From the construction of $M$ given in the Appendix, it is easy to check that we can always choose suitable $\widehat{Z}_1$ and $\widehat{Z}_2$ for which the condition (\ref{detcond}) is satisfied.  
\end{enumerate}

\section{Linearization of Two-Parameter Quadratic Eigenvalue Problem via Newton Basis}

The quadratic two-parameter eigenvalue problem via Newton basis is concerned with finding a
pair $(\lambda,\mu) \in \mathbb{C} \times \mathbb{C}$ and nonzero vectors 
$x_i \in \mathbb{C}^{p_i}$ for which
\begin{equation} \label{qteuan}
Q_{{N}_i}(\lambda,\mu)x_i = 0, \quad i=1,2,
\end{equation}
where
\begin{equation}\label{twopara}
Q_{{N}_i}(\lambda,\mu) =   F_{i} n_{2}(\lam)+ E_{i}n_{1}(\lambda)m_{1}(\mu) + D_{i} m_{2}(\mu)+C_{i} n_1(\lambda)+ B_i m_1(\mu)+A_{i},
\end{equation}
with $A_i,B_i,\ldots,F_i \in \mathbb{C}^{p_i \times p_i}$. 
The pair $(\lambda,\mu)$ is called an eigenvalue of (\ref{twopara}) and 
$x_1 \otimes x_2$ is called the corresponding eigenvector. 
The spectrum of a quadratic two-parameter eigenvalue problem in Newton basis is the set
\begin{equation}
\sigma_{Q_{N}} := \{(\lambda,\mu) \in \mathbb{C}\times\mathbb{C} : 
\det Q_{{N}_i}(\lambda,\mu) = 0, \; i=1,2 \}.
\end{equation}

In the generic case, we observe that (\ref{qteuan}) has $4p_1p_2$ eigenvalues by using 
the following theorem.

\begin{theorem}\cite{cox97}
Let $f(x,y) = g(x,y) = 0$ be a system of two polynomial equations in two unknowns. If it has only finitely many common complex zeros 
$(x,y) \in \mathbb{C}\times\mathbb{C}$, then the number of those zeros is at most 
$\deg(f)\cdot \deg(g)$.
\end{theorem}

The usual approach to solving (\ref{qteuan}) is to linearize it as a two-parameter
eigenvalue problem given by
\begin{equation}\label{tqtp}
L_{{N}_1}(\lambda, \mu)w_1 = \big( A^{(1)}\Gamma_2 + B^{(1)} \widetilde{\Gamma}_2 +C^{(1)}\big) w_1 = 0, \end{equation}
\begin{equation}\label{tqtp1}
L_{{N}_2}(\lambda, \mu)w_2 = \big( A^{(2)} \Gamma_2 + B^{(2)} \widetilde{\Gamma}_2 +C^{(2)}\big) w_2 = 0, 
\end{equation}
where {\scriptsize $$\Gamma_2(\lambda) = \begin{pmatrix}
   \gamma_{2}(\lambda)\otimes I_{p_1} & & \\
   & \gamma_{1}(\lambda)\otimes I_{p_1}& & \\
   & & & \gamma_{1}(\lambda) \otimes I_{p_1}\\
\end{pmatrix}   
\widetilde{\Gamma}_2(\mu) = \begin{pmatrix}
   \widetilde{\gamma}_{1}(\mu)\otimes I_{p_1} & & \\
   & \widetilde{\gamma}_{2}(\mu)\otimes I_{p_1}& & \\
   & & & \widetilde{\gamma}_{1}(\mu)\otimes I_{p_1}\\
\end{pmatrix}$$} for (\ref{tqtp}), 

{\scriptsize $$\Gamma_2(\lambda) = \begin{pmatrix}
   \gamma_{2}(\lambda)\otimes I_{p_2} & & \\
   & \gamma_{1}(\lambda)\otimes I_{p_2}& & \\
   & & & \gamma_{1}(\lambda) \otimes I_{p_2}\\
\end{pmatrix}   
\widetilde{\Gamma}_2(\mu) = \begin{pmatrix}
   \widetilde{\gamma}_{1}(\mu)\otimes I_{p_2} & & \\
   & \widetilde{\gamma}_{2}(\mu)\otimes I_{p_2}& & \\
   & & & \widetilde{\gamma}_{1}(\mu)\otimes I_{p_2}\\
\end{pmatrix}$$} for (\ref{tqtp1}) and $A^{(i)},B^{(i)},C^{(i)} \in \mathbb{C}^{k_i \times k_i}$ with 
$k_i \ge 2p_i, \; i=1,2$, and $w_i = N \otimes x_i$. 
A pair $(\lambda, \mu)$ is called an eigenvalue of (\ref{tqtp}) and (\ref{tqtp1}) if 
$L_{N{i}}(\lambda,\mu)w_i = 0$ for a nonzero vector $w_i$ for $i=1,2$, and 
$w_1 \otimes w_2$ is the corresponding eigenvector. Thus the spectrum
of the linearized two-parameter eigenvalue problem is given by
\begin{equation}
\sigma_{L_{N}} := \{(\lambda,\mu) \in \mathbb{C}\times\mathbb{C} : 
\det L_{{N}_i}(\lambda,\mu) = 0, \; i=1,2 \}.
\end{equation}

Therefore,  the problem (\ref{tqtp}) has $k_1k_2 \ge 4p_1p_2$ 
eigenvalues.

A standard approach to solve a two-parameter eigenvalue problem (\ref{tqtp}) is
by converting it into a coupled generalized eigenvalue problem given by
\[
\Delta_1 z = \Gamma_2 \Delta_0 z, \qquad 
\Delta_2 z = \widetilde{\Gamma_2} \Delta_0 z,
\]
where $z = w_1 \otimes w_2$ and
\begin{align*}
\Delta_0 &= B^{(1)} \otimes C^{(2)} - C^{(1)} \otimes B^{(2)}, \\
\Delta_1 &= C^{(1)} \otimes A^{(2)} - A^{(1)} \otimes C^{(2)}, \\
\Delta_2 &= A^{(1)} \otimes B^{(2)} - B^{(1)} \otimes A^{(2)}.
\end{align*}

The two-parameter eigenvalue problem is called \emph{singular} (resp. \emph{nonsingular}) 
if $\Delta_0$ is singular (resp. nonsingular), see~\cite{AMBPle10}.  

As mentioned earlier, we are interested in finding linear two-parameter polynomials 
$L_{{N}_{i}}(\lambda,\mu)$ for a given quadratic two-parameter eigenvalue problem in Newton basis (\ref{qteuan}) 
such that $\sigma_{Q_{N}} = \sigma_{L_{N}}$. Thus we have the following definition.  

\begin{definition}
Let (\ref{qteuan}) be a quadratic two-parameter eigenvalue problem in Newton basis. A two-parameter 
eigenvalue problem (\ref{tqtp}) and (\ref{tqtp1}) is said to be a \emph{linearization} of (\ref{qteuan}) if 
$L_{N_{i}}(\lambda,\mu)$ is a linearization of $Q_{N_{i}}(\lambda,\mu)$.  
\end{definition}

Thus, if we consider a linearization of a quadratic two-parameter eigenvalue problem in Newton basis, 
then $\sigma_{Q_{N}} = \sigma_{L_{N}}$ is guaranteed. It is also easy to observe that 
$x_1 \otimes x_2$ is an eigenvector corresponding to an eigenvalue $(\lambda,\mu)$ 
of a quadratic two-parameter eigenvalue problem if and only if 
$w_1 \otimes w_2$ is an eigenvector corresponding to the eigenvalue $(\lambda,\mu)$ 
of the linearization.

Using the construction of linearizations for a two-parameter quadratic matrix polynomial in Newton basis described in Section 3, we develop linearizations for (\ref{qteuan}).

\begin{theorem}\label{sp}
Let (\ref{qteuan}) be a quadratic two-parameter eigenvalue problem in Newton basis. A class of linearizations 
of (\ref{qteuan}) is given by
\[
L_{N_{i}}(\lambda,\mu) w_i = 
\big(A^{(i)} + \Gamma_{2} B^{(i)} + \widetilde{\Gamma_{2}} C^{(i)}\big) w_i = 0, 
\qquad w_i = N \otimes x_i, \; i=1,2,
\]
where
\begin{align*}
A^{(i)} &= Z^{(i)}_1 + Z^{(i)}_2 + e_1 \otimes A_i, \\[4pt]
B^{(i)} &=  e_1 \otimes D_i - Y^{(i)}_1 
          +  e_1 \otimes E_i - Z^{(i)}_1 
          +  e_1 \otimes B_i, \\[4pt]
C^{(i)} &= Y^{(i)}_1 + e_1 \otimes F_i - Z^{(i)}_2 
          +  e_1 \otimes C_i,
\end{align*}
\[
Y^{(i)}_1 =
\begin{pmatrix}
Y^{(i)}_{11} \\[4pt] 0 \\[4pt] 0
\end{pmatrix},
\qquad
Z^{(i)}_1 =
\begin{pmatrix}
Z^{(i)}_{11} \\[4pt] Z^{(i)}_{21} \\[4pt] Z^{(i)}_{31}
\end{pmatrix},
\qquad
Z^{(i)}_2 =
\begin{pmatrix}
Z^{(i)}_{12} \\[4pt] Z^{(i)}_{22} \\[4pt] Z^{(i)}_{32}
\end{pmatrix}
\in \mathbb{C}^{3p_i \times p_i},
\]
and
\[
\det 
\begin{pmatrix}
Z^{(i)}_{21} & Z^{(i)}_{22} \\
Z^{(i)}_{31} & Z^{(i)}_{32}
\end{pmatrix}
\neq 0.
\]
\end{theorem}

\begin{proof}
Consider the linearizations 
$L_{N_{i}}(\lambda,\mu)=A^{(i)} + \Gamma_2 B^{(i)} + \widetilde{\Gamma_2} C^{(i)}$ of $Q_{N_{i}}(\lambda,\mu)$, $i=1,2$, 
associated with ansatz vector $0 \neq e_1 \in \mathbb{C}^3$, 
given by Theorem~\ref{lqn}. This completes the proof.  
\end{proof}

We now demonstrate that the linearizations for a quadratic two-parameter eigenvalue problem, as described in Theorem~\ref{sp}, are singular. The following theorem is crucial for the subsequent discussion.

\begin{theorem}\cite{haj}\label{dbt}
The determinant of a block-triangular matrix is the product of the determinants of 
the diagonal blocks.
\end{theorem}

Now we present the following result, whose proof directly follows from Theorem~3.5 in \cite{BA12}. However, for completeness, we provide the proof here.

\begin{theorem}
The linearizations for (\ref{qteuan}) derived in Theorem 4.3 are singular linearizations.
\end{theorem}

\begin{proof}
Consider the linearizations 
\[
L_{N_{i}}(\lambda,\mu)w_i = (A^{(i)} + \Gamma_2 B^{(i)} + \widetilde{\Gamma_2} C^{(i)}) w_i = 0, 
\quad i=1,2,
\]
of $Q_{N_{i}}(\lambda,\mu)$, where
\[
B^{(i)} =
\begin{pmatrix}
 D_i & - Y^{(i)}_{11} +  E_i & - Z^{(i)}_{11} + B_i  \\
0 & 0 & -Z^{(i)}_{21} \\
0 & 0 & -Z^{(i)}_{31}
\end{pmatrix}, 
\qquad
C^{(i)} =
\begin{pmatrix}
Y^{(i)}_{11} & F_i & - Z^{(i)}_{12} +  C_i  \\
0 & 0 & -Z^{(i)}_{22} \\
0 & 0 & -Z^{(i)}_{32}
\end{pmatrix}.
\]

Consequently, we have
\begin{align*}
\Delta_0 &= B^{(1)} \otimes C^{(2)} - C^{(1)} \otimes B^{(2)}.\\
& =
\begin{pmatrix}
 D_1 \otimes C^{(2)} & (-Y^{(1)}_{11}+ E_1)\otimes C^{(2)} & (-Z^{(1)}_{11}+ B_1)\otimes C^{(2)} \\
0 & 0 & -Z^{(1)}_{21} \otimes C^{(2)} \\
0 & 0 & -Z^{(1)}_{31} \otimes C^{(2)}
\end{pmatrix} \\
&-
\begin{pmatrix}
Y^{(1)}_{11}\otimes B^{(2)} &  F_1\otimes B^{(2)} & (-Z^{(1)}_{12}+ C_1)\otimes B^{(2)} \\
0 & 0 & -Z^{(1)}_{22}\otimes B^{(2)} \\
0 & 0 & -Z^{(1)}_{32}\otimes B^{(2)}
\end{pmatrix}. 
\end{align*}

Note that $\Delta_0$ is a block-triangular matrix with one of the diagonal 
blocks equal to $0$. Hence, by Theorem~\ref{dbt}, we have $\det \Delta_0=0$. 
This completes the proof.
\end{proof}

\begin{remark}
Note that given a quadratic two-parameter eigenvalue problem in Newton basis (\ref{qteuan}), we choose 
linearizations $L_i(\lambda,\mu)$ of $Q_{N_{i}}(\lambda,\mu)$ associated with the ansatz 
vector $0 \neq e_1 \in \mathbb{C}^3$, and constructed linearizations 
$L_{N_{i}}(\lambda,\mu)w_i=0$, $w_i=N \otimes x_i$ of (\ref{qteuan}). However, we can 
derive a large class of singular linearizations by choosing linearizations 
$L_{N_{i}}(\lambda,\mu)$ of $Q_{N_{i}}(\lambda,\mu)$ associated with an ansatz vector 
$0\neq v_i \in \mathbb{C}^3$, as described in Section~3.
\end{remark}

\section{Conclusion}
Given a quadratic two-parameter matrix polynomial in Newton basis $Q_{N}(\lam, \mu)$, we constructed a vector space of linear two-parameter matrix polynomials and identify a set of linearizations which lie in the vector space. We also described construction of each of these linearizations. Further, by employing these linearizations, we identify a class of singular linearizations for a quadratic two-parameter eigenvalue problem in Newton basis.  

\section*{Compliance with Ethical Standards} 
\begin{itemize}
    \item {\bf Conflict of interests:} The authors declare no conflict of interest.
\item {\bf Funding:} The authors did not receive support from any organization for the submitted work.
\item {\bf Financial and non-financial interests:} The authors have no relevant financial or non-financial
interests to disclose.\\

Namita Behera \\
Department of Mathematics\\
sikkim University\\
Tadong\\
Sikkim-737102\\
India\\
Email: nbehera@cus.ac.in
\end{itemize}

\newpage

\section*{Appendix} \cite{BA12}

Let 
\[
e_1 = 
\begin{pmatrix}
1 \\
0 \\
0
\end{pmatrix}.
\]
Given a vector 
\[
v = 
\begin{pmatrix}
a \\
b \\
c
\end{pmatrix} \in \mathbb{C}^3,
\]
we can always pick a nonsingular matrix \( M \in \mathbb{C}^{3 \times 3} \) for which \( Mv = e_1 \), as follows:

\[
M =
\begin{cases}
\begin{pmatrix}
1/a & 0 & 0 \\
1/a & -1/b & 0 \\
1/a & 0 & -1/c
\end{pmatrix}, & \text{if } a \neq 0, b \neq 0, c \neq 0 \\[10pt]

\begin{pmatrix}
0 & 1/b & 0 \\
0 & -1/b & 1/c \\
1 & 0 & 0
\end{pmatrix}, & \text{if } a = 0, b \neq 0, c \neq 0 \\[10pt]

\begin{pmatrix}
1 & 1 & 1/c \\
1 & 1 & 0 \\
0 & 1 & 0
\end{pmatrix}, & \text{if } a = 0, b = 0, c \neq 0 \\[10pt]

\begin{pmatrix}
1/a & 0 & 0 \\
0 & 1 & 0 \\
-1/a & 0 & 1/c
\end{pmatrix}, & \text{if } a \neq 0, b = 0, c \neq 0 \\[10pt]

\begin{pmatrix}
1/a & 0 & 0 \\
0 & 1 & 0 \\
0 & 1 & 1
\end{pmatrix}, & \text{if } a \neq 0, b = 0, c = 0 \\[10pt]

\begin{pmatrix}
1/a & 0 & 1 \\
1/a & -1/b & 1 \\
-1/a & 1/b & 0
\end{pmatrix}, & \text{if } a \neq 0, b \neq 0, c = 0 \\[10pt]

\begin{pmatrix}
1 & 1/b & 0 \\
1 & 0 & 0 \\
1 & 0 & 1
\end{pmatrix}, & \text{if } a = 0, b \neq 0, c = 0 \\[10pt]

\begin{pmatrix}
1/a & 0 & 0 \\
1/a & 0 & -1/c \\
0 & 1 & 0
\end{pmatrix}, & \text{if } a \neq 0, b = 0, c \neq 0
\end{cases}
\]


\begin{thebibliography}{10}

\bibitem{VA72} {\sc F. V. Atkinson}, {\em Multiparameter eigenvalue problems}, {Academic Press, New York, 1972.}

\bibitem{BA12} {\sc B.Adhikari}, {\em Vector space of linearizations for the quadratic two parameter matrix polynomial,} {Linear  Multilinear Algebra, 61, (2013), pp. 603--616}. 



\bibitem{AV04}{\sc E.N. Antoulas and S. Vologiannidis}, {\em A new family of companion forms of polynomial matrices,}  {Electron. J. Linear Algebra, 11 (2004), pp. 78--87.}



\bibitem{cox97}{\sc D.Cox, J.Little and D.O’shea}{Using algebraic geometry,Springer
 Verlag,1997.}



\bibitem{defez2002matrix} {\sc E. Defez and A. Law and J. Villanueva-Oller and R. J. Villanueva}, {\em Matrix Newton interpolation and progressive 3D imaging: PC-based computation}, {Mathematical and Computer Modelling, 35 (2002), pp. 303--322}.


\bibitem{TDM}{\sc F. De Ter{\'a}n, F. M. Dopico, and D. S. Mackey}, {\em Fiedler companion linearizations and the recovery of minimal indices,} { SIAM J. Matrix Anal. Appl., 31(2009/10), pp. 2181--2204.}











\bibitem{effenberger2012chebyshev} {\sc C. Effenberger and D. Kressner}, {\em Chebyshev interpolation for nonlinear eigenvalue problems}, {BIT Numerical Mathematics, 52 (2012), pp. 933--951. }
  



\bibitem{HFPSorthop} {\sc H. Faßbender, P. Saltenberger,}{\em On vector spaces of linearizations for matrix polynomials in orthogonal bases,}{ Linear Algebra Appl. 525 (2017), pp. 59--83.}

\bibitem{gohberg82} {\sc I. Gohberg,  P. Lancaster, and L. Rodman,} {\em Matrix polynomials,}{Academic Press Inc., New York  London, 1982.}



\bibitem{guttel2014nleigs} {\sc S. Güttel and R. Van Beeumen and K. Meerbergen and W. Michiels},{\em {NLEIGS}: A class of robust fully rational Krylov methods for nonlinear eigenvalue problems}, {SIAM Journal on Scientific Computing,  36 (2014), pp. 2842--2864. }


\bibitem{FGanttheorymatr98} {\sc F. Gantmacher,} {\em The Theory of Matrices,}{ AMS Chelsea, Providence, RI, 1998.}


\bibitem{haj}{\sc Horn, R. A., and Johnson, C. R.}  {\em Matrix Analysis (2nd ed.)},{ Cambridge University Press, 2012.}




\bibitem{VBKha97} {\sc V. B. Khazanov},{\em The multiparameter eigenvalue problem: Jordan vector semilattices}{ Journal of Mathematical Sciences 86 (1997), pp. 2944--2949.}

\bibitem{VBkha07} {\sc V. B. Khazanov}, {\em To solving spectral problems for multiparameter polynomial matrices}{ Journal of Mathematical Sciences 141 (2007), pp. 1690--1700.} 


\bibitem{TKosir94} {\sc T. Košir} {\em Finite dimensional multiparameter spectral theory: the nonderogatory case}, {Linear Algebra Appl., 212-213 (1994), pp. 45--70.}


\bibitem{VNKub98} {\sc V. N. Kublanovskaya}, {\em An approach to solving multiparameter algebraic problems} {Journal of Mathematical Sciences 89 (1998), pp. 1715--1749.}





\bibitem{mmmm06}{\sc D. S. Mackey, N. Mackey, C. Mehl, and V. Mehrmann}, {\em Vector spaces of linearizations for matrix polynomials,}{ SIAM J. Matrix Anal.
Appl., 28 (2006), pp. 971--1004.}



\bibitem{AMBPle09} {\sc A. Muhic and B. Plestenjak}, {\em On the singular two-parameter eigenvalue problem,} {Electron. J. Linear Algebra 18 (2009), pp. 420--437.}


\bibitem{AMBPle10} {\sc A. Muhic and B. Plestenjak}, {\em On the quadratic two-parameter eigenvalue problem and its linearization}, {Linear Algebra Appl. 432 (2010), pp. 2529--2542.}



\bibitem{VPDSM18} {\sc V. Perović and D. Steven Mackey}, {\em Linearizations of matrix polynomials in Newton bases}, {Linear Algebra Appl., 556 (2018), pp. 1--45}. 



\bibitem{vanbeeumen2013rational} {\sc R. Van Beeumen and K. Meerbergen and W. Michiels}, {\em A rational Krylov method based on Hermite interpolation for nonlinear eigenvalue problems}, {SIAM Journal on Scientific Computing, 35 (2013), pp. 327--350.}


\bibitem{vanbeeumen2015linearization} {\sc R. Van Beeumen and K. Meerbergen and W. Michiels}, {\em Linearization of Lagrange and Hermite interpolating matrix polynomials}, {IMA Journal of Numerical Analysis, 35 (2015), pp. 909--930.}








\end{thebibliography}
\end{document}